\documentclass[a4paper,11pt]{amsart}

\usepackage[cp1250]{inputenc}
\usepackage{amsmath}                            
\usepackage{amssymb}
\usepackage{amsfonts}
\usepackage{mathrsfs}
\usepackage{cite}
\usepackage{textcomp}
\usepackage{array}
\usepackage{multirow}
\usepackage{float}
\usepackage{graphicx}
\usepackage{setspace}
\usepackage{caption}

\newcommand{\cc}{\mathbf{\mathbb{C}}}

\newcommand{\zz}{\mathbf{\mathbb{Z}}}

\newcommand{\supp}{\operatorname{supp}}

\newcommand{\Vol}{\operatorname{Vol}}

\newtheorem{tw}{\textsc{Theorem}}[section]

\newtheorem{wn}[tw]{\textsc{Corollary}}
\newtheorem{lem}[tw]{\textsc{Lemma}}

\theoremstyle{definition}
\newtheorem{df}[tw]{\textsc{Definition}}

\numberwithin{equation}{section}

\newcolumntype{M}[1]{>{\raggedright}m{#1}}

\DeclareMathAlphabet{\mathpzc}{OT1}{pzc}{m}{it}

\allowdisplaybreaks[4]

\begin{document}
\title{L\^e numbers and Newton diagram}
\author{Christophe Eyral, Grzegorz Oleksik, Adam R\'o\.zycki}

\begin{abstract}
We give an algorithm to compute the L\^e numbers of (the germ of) a Newton non-degenerate complex analytic function $f\colon(\cc^n,0) \rightarrow (\cc,0)$ in terms of certain invariants attached to the Newton diagram of the function $f+z_1^{\alpha_1}+\cdots +z_d^{\alpha_d}$, where $d$ is the dimension of the critical locus of $f$ and $\alpha_1,\ldots, \alpha_d$ are sufficiently large integers. This is a version for non-isolated singularities of a famous theorem of A.~G.~Kouchnirenko. As a corollary, we obtain that Newton non-degenerate functions with the same Newton diagram have the same L\^e numbers.
\end{abstract}

\maketitle

\section{Introduction}

The most important numerical invariant attached to a complex analytic function $f\colon(\cc^n,0) \rightarrow (\cc,0)$ with an isolated singularity at $0$ is its Milnor number at $0$ (denoted by $\mu_f(0)$). By a theorem of A.~G.~Kouchnirenko~\cite{K1}, we know that if $f$ is (Newton) non-degenerate and such that its Newton diagram meets each coordinate axis (so-called ``convenient'' function), then $f$ has an isolated singularity at $0$ and $\mu_f(0)$ coincides with the Newton number $\nu(f)$ of $f$ --- a numerical invariant  attached to the Newton diagram of $f$. Actually, if $f$ is a non-degenerate function with an isolated singularity at $0$, then $\mu_f(0)=\nu(f)$ even if $f$ is not convenient (see \cite{BO}). This provides an  elegant and easy way to compute the Milnor number of such functions. 

For a function with a \emph{non-isolated} singularity at~$0$, the Milnor number is no longer relevant. However, we can attach to such a function a series of polar invariants which plays a similar role to that of the Milnor number for an isolated singularity. These polar invariants are called \emph{L\^e numbers}. They were introduced by D.~B.~Massey in the 1990s (see \cite{M3,M4,M}). Then we may wonder whether like for the Milnor number, the L\^e numbers of a non-degenerate function $f$ with a non-isolated singularity at $0$ can also be described with the help of invariants attached to a Newton diagram. In this paper, we positively answer this question. More precisely, we show that the L\^e numbers of a non-degenerate function $f$ can be expressed in terms of certain invariants (which we shall call \emph{modified Newton numbers}) attached to the Newton diagram of the function $f+z_1^{\alpha_1}+\ldots +z_d^{\alpha_d}$, where $d$ is the dimension at $0$ of the critical locus of $f$ and $\alpha_1,\ldots, \alpha_d$ are sufficiently large integers (see Theorem \ref{mt}). 

As an important corollary, we obtain that non-degenerate functions with the same Newton diagram have the same L\^e numbers (see Corollary \ref{cor2}). In particular, any $1$-parameter deformation family of non-degenerate functions with constant Newton diagram has constant L\^e numbers. We recall that families with constant L\^e numbers satisfy remarkable properties. For example, in \cite{M4}, Massey proved that under appropriate conditions the diffeomorphism type of the Milnor fibrations associated to the members of such a family is constant. In \cite{FB}, J.~Fern\'andez de Bobadilla showed that in the special case of families of $1$-dimensional singularities, the constancy of L\^e numbers implies the topological triviality of the family at least if $n\geq 5$. 

The paper is organized as follows. In Section \ref{sect-ln}, we recall the definition of the L\^e numbers. In Section \ref{ndrnn}, following Kouchnirenko's definition of the Newton number, we introduce our modified Newton numbers. Our main result --- the formulas for the L\^e numbers of a non-degenerate function $f$ in terms of the modified Newton numbers of the function $f+z_1^{\alpha_1}+\cdots +z_d^{\alpha_d}$ --- is given in Section \ref{sect-mt}. Corollaries of these formulas are given in Section~\ref{sect-corollaries}. In Section \ref{sect-ex}, we discuss a complete example. Finally, in Sections \ref{pmt} and \ref{pcor2}, we give the proofs of our main result and main corollary respectively.

\section{L\^e numbers}\label{sect-ln}

L\^e numbers are intersection numbers of certain analytic cycles --- so-called \emph{L\^e cycles} --- with certain affine subspaces. The L\^e cycles are defined using the notion of gap sheaf.  In this section, we briefly recall these definitions which are essential for the paper. We also recall the notion of ``polar ratio'' which is involved in so-called \emph{Iomdine-L\^e-Massey formula}. This formula plays a crucial role in the proof of our main theorem.

We follow the presentation given by Massey in \cite{M3,M4,M}.

\subsection{Gap sheaves}
Let $(X,\mathscr{O}_X)$ be a complex analytic space,  $W\subseteq X$ be an analytic subset of $X$, and $\mathscr{I}$ be a coherent sheaf of ideals in $\mathscr{O}_X$. As usual, we denote by $V(\mathscr{I})$ the analytic space defined by the vanishing of $\mathscr{I}$. At each point $x\in V(\mathscr{I})$, we want to consider scheme-theoretically those components of $V(\mathscr{I})$ which are not contained in~$W$. 
For this purpose, we look at a minimal primary decomposition of the stalk $\mathscr{I}_x$ of $\mathscr{I}$ in the local ring $\mathscr{O}_{X,x}$, and we consider the ideal $\mathscr{I}_x \lnot W$ in $\mathscr{O}_{X,x}$ consisting of the intersection of those (possibly embedded) primary components $Q$ of $\mathscr{I}_x$ such that $V(Q)\nsubseteq W$.  
This definition does not depend on the choice of the minimal primary decomposition of $\mathscr{I}_x$. Now, if we perform the operation described above at the point $x$ simultaneously at all points of $V(\mathscr{I})$, then we obtain a coherent sheaf of ideals called a \emph{gap sheaf} and denoted by $\mathscr{I}\lnot W$. Hereafter, we shall denote by $V(\mathscr{I})\lnot W$ the scheme (i.e., the complex analytic space) $V(\mathscr{I}\lnot W)$ defined by the vanishing of the gap sheaf $\mathscr{I}\lnot W$. 

\subsection{L\^e cycles and L\^e numbers}\label{sect-lcln}
Consider an analytic function $f\colon (U,0)\rightarrow(\mathbb{C},0)$, where $U$ is an open neighbourhood of $0$ in $\mathbb{C}^n$, and fix a system of linear coordinates $z=(z_1,\ldots,z_n)$ for $\mathbb{C}^n$. Let $\Sigma f$ be the critical locus of $f$. For $0\leq k\leq n-1$, the $k$th (relative) \emph{polar variety} of $f$ with respect to the coordinates $z$ is the scheme
\begin{equation*}
\Gamma_{f,z}^k:=V\bigg(\frac{\partial f}{\partial z_{k+1}},\ldots,\frac{\partial f}{\partial z_{n}}\bigg) \lnot \Sigma f.
\end{equation*}
The analytic cycle
\begin{equation*}
[\Lambda^k_{f,z}]:=\bigg[\Gamma_{f,z}^{k+1}\cap V\bigg(\frac{\partial f}{\partial z_{k+1}}\bigg)\bigg] - \bigg[\Gamma_{f,z}^k\bigg]
\end{equation*}
is called the $k$th \emph{L\^e cycle} of $f$ with respect to the coordinates $z$. (We always use brackets $[\cdot]$ to denote analytic cycles.)
The $k$th \emph{L\^e number} $\lambda^k_{f,z}(0)$ of $f$ at $0\in\mathbb{C}^n$ with respect to the coordinates $z$ is defined to be the intersection number 
\begin{equation}\label{def-ln}
\lambda^k_{f,z}(0):=\big([\Lambda^k_{f,z}]\cdot [V(z_1,\ldots,z_k)]\big)_0
\end{equation}
provided that this intersection is $0$-dimensional or empty at $0$; otherwise, we say that $\lambda^k_{f,z}(0)$ is \emph{undefined}.\footnote{As usual, $[V(z_1,\ldots,z_k)]$ denotes the analytic cycle associated to the analytic space defined by $z_1=\cdots=z_k=0$. The notation $\big([\Lambda^k_{f,z}]\cdot [V(z_1,\ldots,z_k)]\big)_0$ stands for the intersection number at $0$ of the analytic cycles $[\Lambda^k_{f,z}]$ and $[V(z_1,\ldots,z_k)]$.}
For $k=0$, the relation (\ref{def-ln}) means
\begin{equation*}
\lambda^0_{f,z}(0) = \big([\Lambda^0_{f,z}]\cdot U \big)_0 = \bigg[\Gamma_{f,z}^{1}\cap V\bigg(\frac{\partial f}{\partial z_{1}}\bigg)\bigg]_0.
\end{equation*}

For any $\dim_0\Sigma f< k\leq n-1$, the L\^e number $\lambda^k_{f,z}(0)$ is always defined and equal to zero. For this reason, we usually only consider the L\^e numbers 
\begin{equation*}
\lambda^0_{f,z}(0),\ldots,\lambda^{\dim_0\Sigma f}_{f,z}(0).
\end{equation*}
Note that if $0$ is an \emph{isolated} singularity of $f$, then $\lambda^0_{f,z}(0)$ (which is the only possible non-zero L\^e number) is equal to the Milnor number $\mu_f(0)$ of $f$ at~$0$.

\subsection{Polar ratios}\label{sect-pr}
As already mentioned above, a key ingredient in the proof of our main result is  the Iomdine-L\^e-Massey formula (see \cite[Theorem 4.5]{M}). Roughly, this formula says that if the L\^e numbers of $f$ at $0$ with respect to $z$ exist and if $d:=\dim_0 \Sigma f\geq 1$, then for any integer $\alpha_1$ large enough, $\dim_0\Sigma (f+z_1^{\alpha_1})=d-1$ and the L\^e numbers of $f+z_1^{\alpha_1}$ at $0$ with respect to the rotated coordinates $z^{(1)}:=(z_2,\ldots,z_n,z_1)$ exist and they can be described in terms of the L\^e numbers $\lambda_{f,z}^k(0)$ of the original function $f$. Moreover, the formula says that any $\alpha_1>\rho_{f,z}(0)$ is suitable, where $\rho_{f,z}(0)$ is the maximum ``polar ratio'' of $f$ at $0$ with respect to $z$.
In this section, we recall the definition of polar ratios (see \cite[Definition 4.1]{M}). 

The notation is as in Section \ref{sect-lcln}.
Suppose that $\dim_{0}\Gamma^{1}_{f,z}=1$. Let $\eta$ be an irreducible component of $\Gamma^{1}_{f,z}$ (with its reduced structure) such that $\dim_{0}(\eta\cap V(z_1))=0$.
The \emph{polar ratio} of $\eta$ at~$0$ is the number defined by
\begin{displaymath}
\frac{\bigl([\eta]\cdot[V(f)]\bigr)_{0}} 
{\bigl([\eta]\cdot[V(z_1)]\bigr)_{0}} =
\frac{\Bigl(\Bigl[\eta\Bigr]\cdot 
\Bigl[V\Bigl(\frac{\partial f}{\partial z_1} \Bigr)\Bigr]\Bigr)_{0}}
{\bigl([\eta]\cdot[V(z_1)]\bigr)_{0}} + 1.
\end{displaymath}
If $\dim_{0}(\eta\cap V(z_1))\not=0$, then we say that the polar ratio of $\eta$ at $0$ is equal to $1$. A polar ratio for $f$ at $0$ with respect to $z$ is any one of the polar ratios at~$0$ of any component of $\Gamma^{1}_{f,z}$.  

For example, if $f$ is a homogeneous polynomial and if $\dim_{0}\Gamma^{1}_{f,z}=1$, then each component of $\Gamma^{1}_{f,z}$ is a line, and hence the polar ratios for $f$ at $0$ with respect to $z$ are all equal to $1$ or to the degree $\deg(f)$ of the polynomial $f$ (see \cite[Remark 4.2]{M}). 

In \cite[Section 3.2]{MS}, M. Morgado and M. Saia gave an upper bound for the maximal polar ratio for a semi-weighted homogeneous arrangement.

\section{Newton diagram and modified Newton numbers}\label{ndrnn}

Let $z:=(z_1,\ldots,z_n)$ be a system of coordinates for $\cc^n$, let $U$ be an open neighbourhood of the origin in $\cc^n$, and let 
\begin{equation*}
f\colon(U,0) \rightarrow (\cc,0),\quad
z\mapsto f(z)=\sum _{\alpha} c_{\alpha} z^{\alpha},
\end{equation*} 
be an analytic function, where $\alpha:=(\alpha_1,\ldots,\alpha_n)\in \mathbb{Z}_+^n$, $c_\alpha\in\mathbb{C}$, and $z^{\alpha}$ is a notation for the monomial $z_1^{\alpha_1}\cdots z_n^{\alpha_n}$.  

\subsection{Newton diagram} Here, the reference is Kouchnirenko \cite{K1}.  
The \emph{Newton polyhedron} $\Gamma_{\!+}(f)$ of $f$ (at the origin and with respect to the coordinates $z=(z_1,\ldots,z_n)$) is
the convex hull in $\mathbb{R}_+^n$ of the set
\begin{displaymath}
\bigcup_{c_\alpha\not=0} (\alpha+\mathbb{R}_+^n).
\end{displaymath}
For any $v \in \mathbb{R}_+^n \setminus \{0\}$, put
\begin{equation*}
\begin{aligned}
& \ell(v,\Gamma_{\!+}(f)):=\min \{ \langle v,\alpha\rangle \, ;\, \alpha \in \Gamma_{\!+}(f) \},\\
& \Delta (v,\Gamma _{\!+}(f)):=\{ \alpha \in \Gamma_{\!+}(f)\, ;\,  \langle v,\alpha\rangle= \ell(v,\Gamma_{\!+}(f))\},
\end{aligned}
\end{equation*}
where $\langle \cdot\,,\cdot\rangle$ denotes the standard inner product in $\mathbb{R}^n$.
A subset $\Delta\subseteq\Gamma_{\!+}(f)$ is called a \emph{face} of $\Gamma_{\!+}(f)$ if there exists $v \in \mathbb{R}_+^n \setminus \{0\}$ such that 
$\Delta =\Delta (v,\Gamma_{\!+}(f))$. The \emph{dimension} of a face $\Delta$ of $\Gamma_{\!+}(f)$ is the minimum of the dimensions of the affine subspaces of $\mathbb{R}^n$ containing $\Delta$.  
The \emph{Newton diagram} (also called \emph{Newton boundary}) of~$f$ is the union of the compact faces of $\Gamma_{\!+}(f)$. It is denoted by $\Gamma(f)$. 
We say that $f$ is \emph{convenient} if the intersection of $\Gamma(f)$ with each coordinate axis of $\mathbb{R}^n_+$ is non-empty (i.e., for any $1\leq i\leq n$, the monomial $z_i^{\alpha_i}$, $\alpha_i\geq 1$, appears in the expression $\sum_\alpha c_\alpha z^\alpha$ with a non-zero coefficient).

For any face~$\Delta\subseteq\Gamma(f)$, define the \emph{face function} $f_{\Delta}$ by
\begin{displaymath}
f_{\Delta}(z) := \sum_{\alpha\in\Delta} c_\alpha z^\alpha.
\end{displaymath} 
We say that $f$ is \emph{Newton non-degenerate} (in short, non-degenerate) on the face $\Delta$ if the equations 
\begin{displaymath}
\frac{\partial f_{\Delta}}{\partial z_1}(z) = \cdots =
\frac{\partial f_{\Delta}}{\partial z_n}(z)=0
\end{displaymath}
have no common solution on $(\cc \setminus \{0\})^n$. We say that 
$f$ is \emph{(Newton) non-degenerate} if it is non-degenerate on every face $\Delta$ of~$\Gamma(f)$. 

\subsection{A bound for non-degeneracy of certain functions}\label{subsect-bnd}
Another important ingredient in the proof of our main theorem is Lemma 3.7 of \cite{BO}. This lemma asserts that if $f$ is a non-degenerate function with a singularity at $0$, then there exists a constant $m(f)$ such that for any $\alpha_i> m(f)$, the function $f+z_i^{\alpha_i}$ is non-degenerate too.
Such a (non unique) number $m(f)$ is defined as follows.
For each face $\Delta\subseteq\Gamma(f)$ with maximal dimension (i.e., $\Delta$ is not contained in any other face), choose a vector $v_\Delta \in \mathbb{R}_+^n \setminus \{0\}$ such that $\Delta =\{ \alpha \in \Gamma_{\!+}(f)\, ;\,  \langle v_\Delta,\alpha\rangle= \ell(v_\Delta,\Gamma_{\!+}(f))\}$, and define
\begin{displaymath}
W:=\bigcup_{{\Delta\subseteq\Gamma(f)}\atop{\mbox{\tiny max dim}}} \{ \alpha \in \mathbb{R}_+^n\, ;\, \langle v_\Delta,\alpha\rangle \leq \ell(v_\Delta,\Gamma_{\!+}(f))\},
\end{displaymath} 
where the union is taken over all maximal dimensional faces $\Delta\subseteq\Gamma(f)$. Clearly, $W$ is a compact set and it intersects each coordinate axis of $\mathbb{R}_+^n$ in a closed interval, say $[0,w_i]$ for some $w_i$. Then define
\begin{displaymath}
m(f):= \max_{1\leq i\leq n} w_i.
\end{displaymath} 
Of course, $m(f)$ depends on the choice of the vectors $v_\Delta$. It is possible to define a ``smallest'' number $m_0(f)$ that also guarantees the non-degeneracy of the functions $f+z_i^{\alpha_i}$ with $\alpha_i>m_0(f)$ (see \cite[Section 2]{F}). However, for our purpose, we shall not need it.

\subsection{Newton number}\label{sect-knn}
Again, the reference for this section is \cite{K1}. Throughout the paper, for any subsets $I \subseteq \{1,\ldots,n\}$ and $X\subseteq\mathbb{R}^n_+$, we shall use the following notation:
\begin{equation*}
X^I := \{(x_1,\ldots,x_n) \in X \, ;\,  x_i=0 \text{ if } i \not\in I\}.
\end{equation*} 
In particular, for any $i\in\{1,\ldots,n\}$, the set $X^{\{i\}}$ is nothing but the intersection of $X$ with the $i$th coordinate axis of $\mathbb{R}^n_+$.

Let $\Gamma_{\!-}(f)$ denote the cone over $\Gamma(f)$ with the origin as vertex. If $f$ is convenient, then the \emph{Newton number} $\nu(f)$ of $f$ is defined by
\begin{equation}\label{KNN}
\nu(f):=\sum _{I \subseteq \{1,\ldots,n\}} (-1)^{n-|I|} |I|! 
\Vol_{|I|}(\Gamma_{\! -}(f)^I),
\end{equation}
where $|I|$ is the cardinality of $I$ and $\Vol _{|I|} (\Gamma_{\! -}(f)^I)$ is the $|I|$-dimensional Euclidean volume of $\Gamma_{\! -}(f)^I$. For $I=\emptyset$, the subset $\Gamma_{\!-}(f)^\emptyset$ reduces to~$\{0\}$, and we set $\Vol_{0} (\Gamma_{\!-}(f)^\emptyset)=1$.  

The Newton number can also be defined even if $f$ is not convenient. More precisely, if $I$ is the non-empty subset of $\{1,\ldots,n\}$ such that $\Gamma(f)$ meets the $i$th coordinate axis of $\mathbb{R}^n_+$ if and only if $i\notin I$, then the Newton number $\nu(f)$ of $f$ is defined as
 \begin{equation*}
\nu(f):=\sup_{m\in\mathbb{Z}_+} \nu \bigg(f+\sum_{i\in I}z_i^m\bigg),
\end{equation*}
where of course the Newton number of the (convenient) function $f+\sum_{i\in I}z_i^m$ is given by \eqref{KNN}.
 
\subsection{Modified Newton numbers}\label{sect-rnn}
Following Kouchnirenko's definition of the Newton number, we now introduce our \emph{modified Newton numbers}. 

Let $I$ be a non-empty subset of $\{1,\ldots,n\}$ such that $\Gamma(f)^I\not=\emptyset$. By \cite[Theorem 1]{Ed}, choose a simplicial decomposition of $\Gamma(f)^I$ in which the vertices of a simplex are $0$-dimensional faces of $\Gamma(f)^I$ (such a decomposition is not unique). The cones spanned by the origin $0\in\mathbb{R}^n$ and such simplexes give a simplicial decomposition
\begin{equation*}
\Xi_I:=\{S_{I,r}\}_{1\leq r\leq r_I}
\end{equation*}
of $\Gamma_{\! -}(f)^I$.
Note that 
\begin{equation}\label{rel-vss}
\Vol_{|I|} (\Gamma_{\! -}(f)^I)=\sum_{S_{I,r}\in \Xi_{I},\, \dim S_{I,r}=|I|} 
\Vol _{|I|} (S_{I,r}).
\end{equation}
Clearly, each simplex $S_{I,r}\subseteq (\mathbb{R}^n_+)^I$ may be identified to a simplex (still denoted by $S_{I,r}$) of $\mathbb{R}^{|I|}$, and with such an identification, the volume $\Vol_{|I|} (S_{I,r})$ of a simplex $S_{I,r}$ with maximal dimension (i.e., with dimension $|I|$) is given by
\begin{equation}\label{volsimplex}
\Vol_{|I|} (S_{I,r}) = \pm \frac{1}{|I|!}\det
\left(\begin{matrix}
0 & S_{I,r;1} & \cdots & S_{I,r;|I|}\\
1 & 1 & \cdots & 1
\end{matrix}
\right),
\end{equation}
where $0,S_{I,r;1},\ldots,S_{I,r;|I|}$ are the column vectors representing the coordinates of the vertices of the simplex $S_{I,r}\subseteq\mathbb{R}^{|I|}$. Note that each such column vector has $|I|$ components, so that the matrix in \eqref{volsimplex} has dimension
$(|I|+1)\times (|I|+1)$.

Let $J$ be another subset of $\{1,\ldots,n\}$. We suppose that for any $i\in J$ the Newton boundary $\Gamma(f)$ meets the $i$th coordinate axis of $\mathbb{R}^n_+$. 
Then to each $i_0 \in \{1,\ldots,n\}$, we associate a subset $\Xi_{I,J,i_0}$ of $\Xi_{I}$ (depending on $I$, $J$ and~$i_0$) as follows. If $i_0\in I\cap J$, then we define $\Xi_{I,J,i_0}$ as the set of all simplexes $S_{I,r}\in\Xi_I$ (as simplexes in $(\mathbb{R}^n_+)^I$) with maximal dimension $|I|$ such that for any $i\in J$ the following property holds true:
\begin{equation*}
S_{I,r}^{\{i\}}=S_{I,r}\cap\Gamma_{\! -}(f)^{\{i\}} \mbox{ is an edge of } S_{I,r} 
\Leftrightarrow i=i_0. 
\end{equation*}
(As usual, by an ``edge'' of a simplex we mean a $1$-dimensional face.)
If $i_0\notin J$ (in particular if $J=\emptyset$) or if $i_0\notin I$, then we set $\Xi_{I,J,i_0}:=\emptyset$.

By definition, if $S_{I,r}$ is a simplex of $\Xi_{I,J,i_0}$, then it has maximal dimension and possesses a vertex with coordinates of the form $(0,\ldots,0,\alpha_{i_0},0,\ldots,0)\in\mathbb{R}^n$ (for some $\alpha_{i_0}$ located at the $i_0$th place). To each such a simplex $S_{I,r}\in\Xi_{I,J,i_0}$, we associate a (unique) ``reduced'' simplex $\widetilde{S}_{I,r}$ defined by the same vertices as those of $S_{I,r}$ with the exception of the vertex $(0,\ldots,0,\alpha_{i_0},0,\ldots,0)$ which we replace by $(0,\ldots,0,1,0,\ldots,0)$. We denote by $\widetilde{\Xi}_{I,J,i_0}$ the set of such reduced simplexes.

By convention, for the next definition and all the statements hereafter, we agree that if $I$ is a non-empty subset of $\{1,\ldots,n\}$ such that $\Gamma(f)^I$ is empty, then the corresponding ``simplicial decomposition'' $\Xi_I$ is the empty set.

\begin{df}\label{def-mnn}
For each $J$, $i_0$, and each collection $\Xi:=\{\Xi_I\}_{I \subseteq \{1,\ldots,n\},\, I\not=\emptyset}$ as above, we define a \emph{modified Newton number} $\widetilde{\nu}_{\Xi,J,i_0} (f)$ for the function $f$ by
\begin{equation*}
\widetilde{\nu}_{\Xi,J,i_0} (f) := \sum_{I \subseteq \{1,\ldots,n\},\, I\ni i_0} \Bigg(\sum_{\widetilde{S}_{I,r}\in \widetilde{\Xi}_{I,J,i_0}}
(-1)^{n-|I|} |I|! \Vol_{|I|} (\widetilde{S}_{I,r})\Bigg).
\end{equation*}
(If $\Xi_I=\emptyset$ or if $i_0\notin J$, then $\Xi_{I,J,i_0}=\widetilde{\Xi}_{I,J,i_0}=\emptyset$, and the corresponding term in the above sum is zero by convention.)
\end{df}

Similarly, we introduce the subset $\Xi_{I,J,0}$ of $\Xi_{I}$ consisting of those simplexes $S_{I,r}\in\Xi_I$ with maximal dimension and such that for any $i\in J$ the intersection $S_{I,r}^{\{i\}}=S_{I,r}\cap\Gamma_{\! -}(f)^{\{i\}}$ is not an edge of $S_{I,r}$.

\begin{df}\label{def-smnn}
For each $J$ and each $\Xi:=\{\Xi_I\}_{I \subseteq \{1,\ldots,n\},\, I\not=\emptyset}$ as above, we define a \emph{special} modified Newton number ${\nu}_{\Xi,J,0} (f)$ for the function $f$ by
\begin{equation*}
{\nu}_{\Xi,J,0} (f) := \sum_{I \subseteq \{1,\ldots,n\},\, I\not=\emptyset} \Bigg(\sum_{{S}_{I,r}\in {\Xi}_{I,J,0}} (-1)^{n-|I|} |I|! \Vol_{|I|} ({S}_{I,r})\Bigg).
\end{equation*}
\end{df}

Let us emphasize the fact that the simplexes involved in the definition of the modified Newton number $\widetilde{\nu}_{\Xi,J,i_0} (f)$ are \emph{reduced} simplexes, while those used to define the special modified Newton number ${\nu}_{\Xi,J,0} (f)$ are \emph{not} reduced.

\section{Formulas for the L\^e numbers of a non-degenerate function}\label{sect-mt}

Let $z:=(z_1,\ldots,z_n)$ be a system of linear coordinates for $\cc^n$, let $U$ be an open neighbourhood of the origin in $\cc^n$, and let $f\colon(U,0) \rightarrow (\cc,0)$ be a \emph{non-degenerate} analytic function. We denote by $\Sigma f$ the critical locus of~$f$, and we suppose that $d:=\dim_0\Sigma f\geq 1$. 
We also assume that the L\^e numbers 
\begin{equation*}
\lambda^0_{f,z}(0),\ldots,\lambda^{d}_{f,z}(0)
\end{equation*}
of $f$ at $0$ with respect to the coordinates $z=(z_1,\ldots,z_n)$ are defined.
For example, if the coordinates are ``prepolar'' for $f$ (see \cite[Definition 1.26]{M}), then the corresponding L\^e numbers do exist. In particular, this is the case if $f$ has an ``aligned'' singularity at $0$ (e.g., a line singularity) and the coordinates are ``aligning'' for $f$ at $0$ (see \cite[Definition 7.1]{M}). 

For any $1\leq q\leq d$, we consider the function
\begin{equation}\label{lffq}
f_q(z):=f(z)+z_1^{\alpha_1}+ \dots +z_q^{\alpha_q},
\end{equation}
where ${\alpha_1},\ldots,{\alpha_q}$ are integers such that, for any $1\leq p\leq q$,
\begin{equation*}
\alpha_{p} > \max \{2,\rho_{f_{p-1},z^{(p-1)}}(0),m(f_{p-1})\}.
\end{equation*}
Here, $\rho_{f_{p-1},z^{(p-1)}}(0)$ is the maximum polar ratio for $f_{p-1}$ at $0$ with respect to the rotated coordinates 
\begin{equation*}
z^{(p-1)}:=(z_p,\ldots,z_n,z_1,\ldots,z_{p-1}),
\end{equation*}
and $m(f_{p-1})$ is a bound which guarantees the non-degeneracy of the function $f_p$ (see Sections \ref{sect-pr} and \ref{subsect-bnd}).
(By $f_0$ and $z^{(0)}$ we~mean~$f$ and $z$ respectively.) 
For example, if $f$ is a homogeneous polynomial such that $d:=\dim_0\Sigma f=1$, then we can take $f_d(z)=f_1(z):=f(z)+z_1^{\alpha_1}$, where $\alpha_1>\mbox{max}\{2, \deg(f)\}$. 

Hereafter, we are mainly interested in the modified Newton numbers of the function $f_d$. For each non-empty subset $I \subseteq \{1,\ldots,n\}$, we choose a simplicial decomposition 
\begin{equation*}
\Xi_I:=\{S_{I,r}\}_{1\leq r\leq r_I}
\end{equation*}
 of $\Gamma_{\! -}(f_d)^I$ as in Section \ref{sect-rnn} (again, if $\Gamma(f_d)^I=\emptyset$, then $\Xi_I=\emptyset$), and we write $\Xi:=\{\Xi_I\}_{I \subseteq \{1,\ldots,n\},\, I\not=\emptyset}$. Since throughout this section we shall only consider modified Newton numbers of the form 
\begin{equation*}
{\nu}_{\Xi,\{1,\ldots,d\},0} (f_d)
\quad\mbox{and}\quad
\widetilde{\nu}_{\Xi,\{1,\ldots,d\},k} (f_d)
\end{equation*}
($1\leq k\leq d$) where $d$ is the dimension at $0$ of the critical locus $\Sigma f$, we may simplify the notation as follows: 
\begin{equation*}
{\nu}_{\Xi,0} (f_d):={\nu}_{\Xi,\{1,\ldots,d\},0} (f_d)
\quad\mbox{and}\quad
\widetilde{\nu}_{\Xi,k} (f_d):=\widetilde{\nu}_{\Xi,\{1,\ldots,d\},k} (f_d).
\end{equation*}

Here is our main result.

\begin{tw}\label{mt}
Suppose that $f$ is non-degenerate, $d:=\dim_0\Sigma f\geq 1$ and the L\^e numbers $\lambda^k_{f,z}(0)$ of $f$ at $0$ with respect to the coordinates $z=(z_1,\ldots,z_n)$ are defined for any $0\leq k\leq d$. Then the following two assertions hold true.
\begin{enumerate}
\item
The modified Newton numbers ${\nu}_{\Xi,0} (f_d)$ and $\widetilde{\nu}_{\Xi,k} (f_d)$
of the function $f_d$ do not depend on the choice of $\Xi:=\{\Xi_I\}_{I \subseteq \{1,\ldots,n\},\, I\not=\emptyset}$. Therefore, we may further simplify the notation as follows: 
\begin{equation*}
{\nu}_{0} (f_d):={\nu}_{\Xi,0} (f_d)
\quad\mbox{and}\quad
\widetilde{\nu}_{k} (f_d):=\widetilde{\nu}_{\Xi,k} (f_d).
\end{equation*}
\item
The L\^e numbers $\lambda^0_{f,z}(0),\ldots,\lambda^{d}_{f,z}(0)$ are given by the following for\penalty 10000 mulas:\vskip 1mm
\begin{enumerate}
\item[$\cdot$]
$\lambda^0_{f,z}(0)=(-1)^n+\nu_0(f_d)+\widetilde{\nu}_1(f_d)$;\vskip 1mm
\item[$\cdot$]
$\lambda^k_{f,z}(0)=(-1)^{k-1}(\widetilde{\nu}_k(f_d)-\widetilde{\nu}_{k+1}(f_d))$
for $1\leq k\leq d-1$ (if~$d\geq \penalty 10000 2$);\vskip 1mm
\item[$\cdot$]
$\lambda^d_{f,z}(0)=(-1)^{d-1}\widetilde{\nu}_d(f_d)$.\vskip 1mm
\end{enumerate}
\end{enumerate}
\end{tw}

Theorem \ref{mt} is a version for non-isolated singularities of the  Kouchnirenko theorem mentioned in the introduction. It will be proved in Section \ref{pmt}. The formulas given in item (2) reduce the calculation of the L\^e numbers of a non-degenerate function to a simple computation of volumes of simplexes. Certainly, these formulas are well suited for computer algebra programs.

\section{Corollaries}\label{sect-corollaries}

Let $z=(z_1,\ldots,z_n)$ be linear coordinates for $\mathbb{C}^n$.
The first important corollary of Theorem \ref{mt} is the invariance of the L\^e numbers within the class of non-degenerate functions with fixed Newton diagram. More precisely we have the following statement.

\begin{wn}\label{cor2}
Let $f,g \colon (U,0) \rightarrow (\cc , 0)$ be two non-degenerate analytic functions, where $U$ is an open neighbourhood of the origin of $\mathbb{C}^n$. Suppose that the dimensions at~$0$ of the critical loci $\Sigma f$ and $\Sigma g$ of $f$ and $g$, respectively, are greater than or equal to $1$.  If furthermore $\Gamma(f)=\Gamma(g)$ and the L\^e numbers of $f$ and $g$ at~$0$ with respect to the coordinates $z=(z_1,\ldots,z_n)$ exist, then $\dim_0\Sigma f=\dim_0\Sigma g$, and for any $0\leq k\leq n-1$, we have
\begin{equation*}
\lambda^k_{f,z}(0) =\lambda^k_{g,z}(0).
\end{equation*}
\end{wn}

Corollary \ref{cor2} will be proved in Section \ref{pcor2}. In particular, it implies that any $1$-parameter deformation family of non-degenerate functions with constant Newton diagram has constant L\^e numbers, provided that these numbers exist. Here is a more precise statement.

\begin{wn}\label{cor3}
Let $\{f_t\}$ be a $1$-parameter deformation family of analytic functions $f_t$ defined in an open neighbourhood of $0\in\mathbb{C}^n$ and depending analytically on the parameter $t\in\mathbb{C}$. If for any sufficiently small $t$ (say $\vert t\vert\leq\varepsilon$ for some $\varepsilon>0$), the function $f_t$ is non-degenerate, $\Gamma(f_t)=\Gamma(f_0)$ and all the L\^e numbers $\lambda_{f_t,z}^k (0)$ are defined, then $\dim_0\Sigma f_t=\dim_0\Sigma f_0$ and $\lambda^k_{f_t,z}(0)=\lambda^k_{f_0,z}(0)$ for all small~$t$.
\end{wn}

By combining Corollary \ref{cor3} with \cite[Theorem 9.4]{M} and \cite[Theorem 42]{FB}, we obtain a new proof of the following result, which is a special case of a much more general theorem of J. Damon \cite{D}.

\begin{wn}[Damon]
Let $\{f_t\}$ be a family as in Corollary \ref{cor3}, that is, such that for any sufficiently small $t$, the function $f_t$ is non-degenerate, $\Gamma(f_t)=\Gamma(f_0)$ and all the L\^e numbers $\lambda_{f_t,z}^k (0)$ are defined. Under these assumptions, the following two assertions hold true.
\begin{enumerate}
\item
If for all small $t$, the coordinates $z=(z_1,\ldots,z_n)$ are prepolar for $f_t$ and $\dim_0\Sigma f_t\leq n-4$, then the diffeomorphism type of the Milnor fibration of $f_t$ at $0$ is independent of $t$ for all small $t$.
\item
If $n\geq 5$ and $\dim_0\Sigma f_t=1$ for all small $t$, then the family $\{f_t\}$ is topologically trivial.
\end{enumerate}
\end{wn}

Indeed, by Corollary \ref{cor3}, the family $\{f_t\}$ has constant L\^e numbers with respect to the coordinates $z=(z_1,\ldots,z_n)$. Item (1) then follows from \cite[Theorem 9.4]{M} while item (2) is a consequence of \cite[Theorem 42]{FB}.

In fact, in \cite{D}, Damon obtains the topological triviality without the restrictions $n\geq 5$ or $\dim_0\Sigma f_t=1$.
A third proof (based on so-called ``uniform stable radius'') of item (2) for line singularities is also given in \cite{E}. 

Finally, combined with \cite[Theorem 3.3]{M}, Theorem \ref{mt} has the following corollary about the Euler characteristic of the Milnor fibre associated to a non-degenerate function.

\begin{wn}\label{mc}
Again, assume that $f$ is non-degenerate, $d:=\dim_0\Sigma f\geq 1$ and the L\^e numbers $\lambda^k_{f,z}(0)$ of $f$ at $0$ with respect to the coordinates $z=(z_1,\ldots,z_n)$ are defined for any $0\leq k\leq d$.
If furthermore the coordinates $z=(z_1,\ldots,z_n)$ are prepolar for $f$, then the reduced Euler characteristic $\widetilde{\chi}(F_{f,0})$ of the Milnor fibre $F_{f,0}$ of $f$ at $0$ is given by
\begin{equation*}
\widetilde{\chi}(F_{f,0})=(-1)^{n-1}(\nu_0(f_d)+(-1)^n),
\end{equation*}
where $f_d$ is defined by \eqref{lffq}.
\end{wn}

Indeed, by \cite[Theorem 3.3]{M}, we have
\begin{equation*}
\widetilde{\chi}(F_{f,0})=\sum_{k=0}^d (-1)^{n-1-k}\lambda^k_{f,z}(0).
\end{equation*}
Thus, to get the formula in Corollary \ref{mc}, it suffices to replace $\lambda^k_{f,z}(0)$ by its expression in terms of the modified Newton numbers given in   Theorem~\ref{mt}.

\section{Example}\label{sect-ex}

Consider the homogeneous polynomial function 
\begin{equation*}
f(z_1,z_2,z_3):=z_1^2z_2^2+z_2^4+z_3^4.
\end{equation*}
The Newton diagram $\Gamma(f)$ of $f$ is nothing but the triangle in $\mathbb{R}^3_+$ (with coordinates $(x_1,x_2,x_3)$) defined by the vertices $A=(2,2,0)$, $B=(0,4,0)$ and $C=(0,0,4)$ (see Figure \ref{Fig1}). We easily check that $f$ is non-degenerate. The critical locus $\Sigma f$ of $f$ is given by the $z_1$-axis, and the restriction of $f$ to the hyperplane $V(z_1)$ defined by $z_1=0$ has an isolated singularity at $0$. In other words, $f$ has a \emph{line singularity} at $0$ in the sense of \cite[\S 4]{M7}. Then, by \cite[Remark 1.29]{M}, the partition of $V(f):=f^{-1}(0)$ given by
\begin{equation*}
\mathscr{S}:=\{V(f)\setminus \Sigma f, \Sigma f\setminus\{0\},\{0\}\}
\end{equation*}
is a ``good stratification'' for $f$ in a neighbourhood of $0$, and the hyperplane $V(z_1)$ is a ``prepolar slice'' for $f$ at $0$ with respect to $\mathscr{S}$ (see \cite[Definitions 1.24 and 1.26]{M}). In other words, the coordinates $z=(z_1,z_2,z_3)$ are prepolar for $f$. In particular, combined with \cite[Proposition 1.23]{M}, this implies that the L\^e numbers $\lambda^0_{f,z}(0)$ and $\lambda^1_{f,z}(0)$ are defined.
We can compute these numbers either using the definition or by applying Theorem \ref{mt}. 

\begin{figure}[t]
\includegraphics[width=16cm,height=5cm]{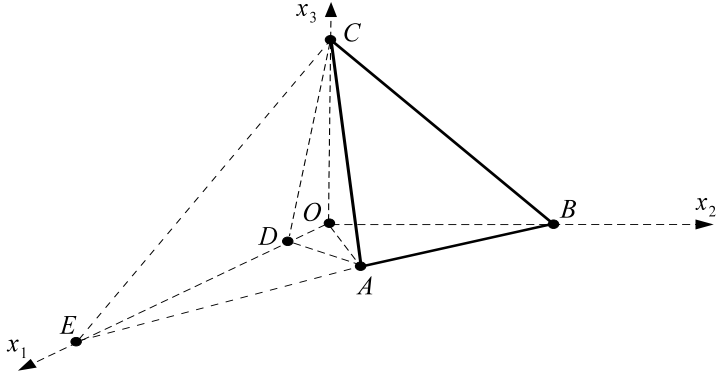}
\caption{Newton diagrams of $f$ and $f_1$}
\label{Fig1}
\end{figure}

\subsection{Calculation using the definition} 
We need to compute the polar varieties $\Gamma^2_{f,z}$ and $\Gamma^1_{f,z}$ and the L\^e cycles $[\Lambda^1_{f,z}]$ and $[\Lambda^0_{f,z}]$. By definition,
\begin{flalign*}
\qquad\Gamma^2_{f,z} = V\bigg(\frac{\partial f}{\partial z_3}\bigg)\lnot V(z_2,z_3)
= V(z_3^3)\lnot V(z_2,z_3)
= V(z_3^3),&&
\end{flalign*}
while
\begin{flalign*}
\qquad\Gamma^1_{f,z} &= V\bigg(\frac{\partial f}{\partial z_2}, \frac{\partial f}{\partial z_3}\bigg)\lnot V(z_2,z_3)\\
&= V(z_2(2z_1^2+4z_2^2),z_3^3)\lnot V(z_2,z_3)
= V(2z_1^2+4z_2^2,z_3^3).&&
\end{flalign*}
It follows that the L\^e cycles are given by
\begin{flalign*}
\qquad[\Lambda^1_{f,z}] &= \bigg[\Gamma^2_{f,z}\cap V\bigg(\frac{\partial f}{\partial z_2}\bigg)\bigg] - \bigg[\Gamma^1_{f,z}\bigg]\\
&= [V(z_3^3)\cap V(z_2(2z_1^2+4z_2^2))] - [V(2z_1^2+4z_2^2,z_3^3)]\\
&= [V(z_2,z_3^3)]&&
\end{flalign*}
and
\begin{flalign*}
\qquad[\Lambda^0_{f,z}] &= \bigg[\Gamma^1_{f,z}\cap V\bigg(\frac{\partial f}{\partial z_1}\bigg)\bigg] \\
&= [V(2z_1^2+4z_2^2,z_3^3)\cap V(z_1z_2^2)] \\
&=[V(z_1,z_2^2,z_3^3)]+[V(z_1^2,z_2^2,z_3^3)].&&
\end{flalign*}
Finally the L\^e numbers are given by
\begin{flalign*}
\qquad\lambda^1_{f,z} &= ([\Lambda^1_{f,z}]\cdot [V(z_1)])_0=[V(z_1,z_2,z_3^3)]_0=3;\\
\qquad\lambda^0_{f,z} &= ([\Lambda^0_{f,z}]\cdot \mathbb{C}^3)_0= 6+12=18.&&
\end{flalign*}

\subsection{Calculation using Theorem \ref{mt}}
Consider a polynomial function 
\begin{equation*}
f_1(z_1,z_2,z_3):=f(z_1,z_2,z_3)+z_1^{\alpha_1}
\end{equation*} 
such that $\alpha_1>\max\{2,\rho_{f,z}(0),m(f)\}$. Since $f$ is a homogeneous polynomial of degree $4$, the maximum polar ratio $\rho_{f,z}(0)$ for $f$ at $0$ with respect to the coordinates $z$ is $4$ (see Section \ref{sect-pr}), and clearly, we can take $m(f)=4$. So, let us take for instance $\alpha_1=5$. Clearly, $\Gamma_{\! -}(f_1)$ is the union of two tetrahedra $\{O,A,C,E\}$ and $\{O,A,B,C\}$. For each subset $I \subseteq \{1,2,3\}$, take the ``natural'' simplicial decomposition $\Xi_I$ of $\Gamma_{\! -}(f_1)^I$ generated by the vertices of the set $\{O,A,B,C,E\}\cap (\mathbb{R}^n_+)^I$ as suggested in Figure \ref{Fig1} (at this level, we ignore the point $D$ mentioned in the figure). For example, $\Xi_{\{1,2\}}$ is defined by the simplexes $\{O,A,E\}$ and $\{O,A,B\}$.
By Theorem \ref{mt},
\begin{equation*}
\lambda^0_{f,z}(0)=(-1)^3+\nu_0(f_1)+\widetilde{\nu}_1(f_1)
\quad\mbox{and}\quad
\lambda^1_{f,z}(0)=(-1)^0\widetilde{\nu}_1(f_1).
\end{equation*}
The data to compute the modified Newton numbers $\nu_0(f_1)$ and $\widetilde{\nu}_1(f_1)$ are given in Table \ref{data}. In this table, $O=(0,0,0)$, $D=(1,0,0)$, and $A$, $B$, $C$ are as above. Each pair in the third and fourth columns of the table consists of a simplex together with its volume. For example, in the first row of the third column, the pair $(\{O,D\};1)$ consists of the simplex $\{O,D\}\in\widetilde{\Xi}_{\{1\},\{1\},1}$ and its volume $\Vol_1(\{O,D\})=1$. The calculation shows that $\nu_0(f_1)=16$ and $\widetilde{\nu}_1(f_1)=3$, and therefore the L\^e numbers are given by
\begin{equation*}
\lambda^0_{f,z}(0)=18
\quad\mbox{and}\quad
\lambda^1_{f,z}(0)=3.
\end{equation*}

\begin{table}[b]
\centering\vskip 8mm
\onehalfspacing 
\begin{tabular}{|c|c|c|c|}
\hline
$I$ 
& $(-1)^{3-|I|}|I|!$ 
& $(S\in\widetilde{\Xi}_{I,\{1\},1} ; \mbox{Vol}_{|I|}(S))$ 
& $(S\in\Xi_{I,\{1\},0} ; \mbox{Vol}_{|I|}(S))$ \\
\hline
$\{1\}$     & $1$  & $(\{O,D\};1)$               & $(\emptyset;0)$\\
$\{2\}$     & $1$  & $(\emptyset;0)$             & $(\{O,B\};4)$\\
$\{3\}$     & $1$  & $(\emptyset;0)$             & $(\{O,C\};4)$\\
$\{1,2\}$   & $-2$ & $(\{O,A,D\};1)$             & $(\{O,A,B\};4)$\\
$\{1,3\}$   & $-2$ & $(\{O,C,D\};2)$             & $(\emptyset;0)$\\
$\{2,3\}$   & $-2$ & $(\emptyset;0)$             & $(\{O,B,C\};8)$\\
$\{1,2,3\}$ & $6$  & $(\{O,A,C,D\};\frac{4}{3})$ & $(\{O,A,B,C\};\frac{16}{3})$\\
\hline
\end{tabular}
\caption{\label{data}Data to compute $\nu_0(f_1)$ and $\widetilde{\nu}_1(f_1)$}
\end{table}

\subsection{Euler characteristic}
Since the coordinates $z=(z_1,z_2,z_3)$ are prepolar for $f$, Corollary \ref{mc} says that to calculate the reduced Euler characteristic $\widetilde{\chi}(F_{f,0})$ of the Milnor fibre $F_{f,0}$ of $f$ at~$0$, it suffices to compute the special modified Newton number $\nu_0(f_1)$. Precisely, $\widetilde{\chi}(F_{f,0})$ is given by
\begin{equation*}
\widetilde{\chi}(F_{f,0})=(-1)^{2}(\nu_0(f_1)+(-1)^3)=15.
\end{equation*}

\section{Proof of Theorem \ref{mt}}\label{pmt}

Applying the Iomdine-L\^e-Massey formula (see \cite[Theorem 4.5]{M}) successively to $f,f_1,\ldots,f_{d-1}$ shows that for any $0\leq q\leq d-1$:
\begin{enumerate}
\item
$\Sigma f_{q+1}=\Sigma f\cap V(z_1,\ldots,z_{q+1})$ in a neighbourhood of the origin;\vskip 1mm
\item
$\dim_0\Sigma f_{q+1}=d-(q+1)$;\vskip 1mm
\item
the L\^e numbers $\lambda^k_{f_{q+1},z^{(q+1)}}(0)$ of $f_{q+1}$  at $0$ with respect to the rotated coordinates 
\begin{align*}
z^{(q+1)}=(z_{q+2},\ldots,z_n,z_1,\ldots,z_{q+1})
\end{align*}
 exist for all $0\leq k\leq d-(q+1)$ and are given by
\begin{align*}
\begin{cases}
\lambda ^0_{f_{q+1},z^{(q+1)}}(0)=\lambda^0_{f_{q},z^{(q)}}(0)+(\alpha_{q+1}-1)\lambda^1_{f_{q},z^{(q)}}(0);\\
\lambda^k_{f_{q+1},z^{(q+1)}}(0)=(\alpha_{q+1}-1)\lambda^{k+1}_{f_{q},z^{(q)}}(0) \quad \mbox{for}\quad  1\leq k\leq d-(q+1);
\end{cases}
\end{align*}
\end{enumerate}
where $\lambda^{k}_{f_{q},z^{(q)}}(0)$ is the $k$th L\^e number of $f_q$ at $0$ with respect to the rotated coordinates 
\begin{equation*}
z^{(q)}=(z_{q+1},\ldots,z_n,z_1,\ldots,z_{q}),
\end{equation*} 
and where $\alpha_{q+1}$ is an integer satisfying
\begin{equation*}
\alpha_{q+1} > \max \{2,\rho_{f_{q},z^{(q)}}(0),m(f_{q})\}.
\end{equation*}
In particular (see \cite[Corollary 4.6]{M}) $f_d$ has an isolated singularity at~$0$ and its Milnor number $\mu_{f_d}(0)$ (which, in this case, coincides with its $0$th L\^e number $\lambda^0_{f_{d},z^{(d)}}(0)$) is given by
\begin{equation}\label{fmnfd}
\mu_{f_d}(0)=\lambda^0 _{f,z}(0)+\sum_{k=1}^d\left(\prod_{q=1}^k (\alpha_q-1)\right) \lambda^k _{f,z}(0).
\end{equation}

Let $\{i_1,\ldots,i_p\}$ be the subset of $\{1,\ldots,n\}\setminus\{1,\ldots,d\}$ consisting of all indices $i$ for which $\Gamma(f_d)$ does not meet the $i$th coordinate axis of $\mathbb{R}^n_+$.
Then, by \cite[Lemmas 3.6--3.8 and Corollary 3.11]{BO} and \cite[Th\'eor\`eme~I]{K1}, for any $0\ll \alpha_1\ll \cdots\ll \alpha_d\ll \alpha_{i_1}\ll\cdots\ll \alpha_{i_p}$ sufficiently large, the function 
\begin{align*}
f'_d(z) & := \underbrace{f(z)+z_{1}^{\alpha_{1}}+\cdots+z_{d}^{\alpha_{d}}}_{f_d(z)}
+z_{i_1}^{\alpha_{i_1}}+\cdots+z_{i_p}^{\alpha_{i_p}}
\end{align*}
is non-degenerate, convenient, and the following equalities hold true:
\begin{equation}\label{newton}
\mu_{f_d}(0)=\mu_{f'_d}(0)=\nu(f'_d)=\nu(f_d).
\end{equation}

The expression \eqref{fmnfd} for the Milnor number $\mu_{f_d}(0)$ can be viewed as a polynomial in the variables $\alpha_1,\ldots,\alpha_d$. Its linear part is given by
\begin{equation}\label{lpmn}
\sum_{k=0}^d (-1)^k \lambda ^k _{f,z}(0) + 
\sum_{i=1}^d \bigg( \alpha_i \sum_{k=i}^d (-1)^{k-1} \lambda ^k _{f,z}(0) \bigg).
\end{equation} 

Now we need the following lemma.

\begin{lem}\label{lemma-nm}
The function $f$ has no term of the form $c_1z_1^{a_1},\ldots,c_dz_d^{a_d}$, where $c_i\in\mathbb{C}\setminus \{0\}$, $a_i\in\mathbb{Z}_{>0}$.
\end{lem}

We postpone the proof of this lemma to the end of this section, and we first complete the proof of Theorem \ref{mt}.

Since $f$ has no term of the form $c_1z_1^{a_1},\ldots,c_dz_d^{a_d}$, the Newton number $\nu (f'_d)$ can be viewed as a polynomial in the variables $\alpha_1,\ldots,\alpha_d$ and $\alpha_{i_1},\ldots,\alpha_{i_p}$. Its linear part with respect to $\alpha_1,\ldots,\alpha_d$ has the form
\begin{equation}\label{lpnn}
P_0(\alpha_{i_1},\ldots,\alpha_{i_p})+\alpha_1\, P_1(\alpha_{i_1},\ldots,\alpha_{i_p}) +\cdots
+\alpha_d\, P_d(\alpha_{i_1},\ldots,\alpha_{i_p}),
\end{equation}
where $P_i(\alpha_{i_1},\ldots,\alpha_{i_p})$ are polynomials in $\alpha_{i_1},\ldots,\alpha_{i_p}$.
Taking the difference $\mu_{f_d}(0)-\nu (f'_d)$ gives a polynomial 
\begin{equation*}
Q(\alpha_1,\ldots,\alpha_d,\alpha_{i_1},\ldots,\alpha_{i_p}):=\mu_{f_d}(0)-\nu (f'_d)
\end{equation*} 
in the variables $\alpha_1,\ldots,\alpha_d,\alpha_{i_1},\ldots,\alpha_{i_p}$. Then it follows from \eqref{newton} that for any $0\ll \alpha_1\ll \cdots\ll \alpha_d\ll \alpha_{i_1}\ll\cdots\ll \alpha_{i_p}$ sufficiently large (equivalently,  for any $(\alpha_1,\ldots,\alpha_d,\alpha_{i_1},\ldots,\alpha_{i_p})$ in the set $Z(d+p)$ which appears in Lemma~\ref{lemma-ul} of the appendix, with the appropriate coefficients $c_1$ and $c_\ell(\alpha_1,\ldots,\alpha_{\ell-1})$ for $2\leq \ell\leq d+p$), we have
\begin{equation*}
Q(\alpha_1,\ldots,\alpha_d,\alpha_{i_1},\ldots,\alpha_{i_p})=0.
\end{equation*}
Thus applying Lemma \ref{lemma-ul} shows that $Q$ identically vanishes.
In particular, comparing the coefficients of the linear parts \eqref{lpmn} and \eqref{lpnn} of $\mu_{f_d}(0)$ and $\nu (f'_d)$, respectively, shows that
the polynomials $P_i:=P_i(\alpha_{i_1},\ldots,\alpha_{i_p})$ are independent of $\alpha_{i_1},\ldots,\alpha_{i_p}$ (i.e., $P_i$ is constant) and are given by
\begin{equation*}
\left\{
\begin{aligned}
P_0 & = \sum_{k=0}^d (-1)^k \lambda ^k _{f,z}(0);\\
P_i & = \sum_{k=i}^d (-1)^{k-1} \lambda^k _{f,z}(0) \quad \mbox{for}\quad  1\leq i\leq d.\\
\end{aligned}
\right.
\end{equation*}
Theorem \ref{mt} is now an immediate consequence of the following lemma.

\begin{lem}\label{fl}
For each non-empty subset $I\subseteq\{1,\ldots,n\}$, choose a simplicial decomposition 
\begin{equation*}
{\Xi}'_I:=\{{S}'_{I,r}\}_{1\leq r\leq r'_I}
\end{equation*}
of $\Gamma_{\! -}({f}'_d)^I$ as in Section \ref{sect-rnn} such that its restriction to $\Gamma_{\! -}(f_d)^I$ coincides with the simplicial decomposition $\Xi_I$.  (We can always achieve this condition by taking $\alpha_{i_1},\ldots,\alpha_{i_p}$ sufficiently large.)
Write $\Xi':=\{{\Xi}'_I\}_{I\subseteq\{1,\ldots,n\},\, I\not=\emptyset}$ and set $J:=\{1,\ldots,d,i_1,\ldots,i_p\}$. Then the following equalities hold true:
\begin{equation*}
\left\{
\begin{aligned}
& P_{i_0}=\widetilde{\nu}_{\Xi',J,i_0}(f'_d)=\widetilde{\nu}_{i_0}(f_d) 
\mbox{ for } 1\leq i_0\leq d;\\
& P_{0}=\nu_{\Xi',J,0}(f'_d)+(-1)^n=\nu_{0}(f_d)+(-1)^n.\\
\end{aligned}
\right.
\end{equation*}
\end{lem}

To complete the proof of Theorem \ref{mt}, it remains to prove Lemmas \ref{lemma-nm} and \ref{fl}. We start with the proof of Lemma \ref{fl}.

\begin{proof}[Proof of Lemma \ref{fl}]
By \eqref{KNN} and \eqref{rel-vss}, the Newton number $\nu(f'_d)$ is (up to coefficients of the form $(-1)^{n-|I|}\, |I|!$) a sum of volumes of the form $\Vol_{|I|}({S}'_{I,r})$, where $\emptyset\not=I \subseteq \{1,\ldots,n\}$ and ${S}'_{I,r}$ is a simplex of ${\Xi}'_I$ with maximal dimension $|I|$, plus the number 
\begin{equation*}
(-1)^{n-|\emptyset|} |\emptyset|!\, \mbox{Vol}_{|\emptyset|}(\Gamma_{\!-}(f_d')^\emptyset)=(-1)^n,
\end{equation*}
which corresponds to $I=\emptyset$ in the definition of $\nu(f'_d)$ (see Section \ref{sect-knn}). If for any $1\leq i_0\leq d$ the matrix used to compute the volume $\Vol_{|I|}({S}'_{I,r})$ (see~\eqref{volsimplex}) does not have any column of the form
\begin{equation*}
\left(\begin{matrix}
\beta_1 & \cdots & \beta_{i_0-1} & \alpha_{i_0}& \beta_{i_0+1} & \cdots & \beta_{|I|} & 1
\end{matrix}
\right)^T,
\end{equation*}
then $\Vol_{|I|}({S}'_{I,r})$ contributes to the term $P_0$ which appears in \eqref{lpnn}. (Here, the letter ``$T$'' stands for the transposed matrix.) On the other hand, if it contains such a column for some $i_0\in\{1,\ldots,d\}$, then necessarily the $\beta_i$'s are zero, and the column is of the form
\begin{equation*}
C_{i_0}:=
\left(\begin{matrix}
0 & \cdots & 0 & \alpha_{i_0}& 0 & \cdots & 0 & 1
\end{matrix}
\right)^T
\end{equation*}
(because $\Gamma(f_d')$ intersects the $i_0$th coordinate axis of $\mathbb{R}^n_+$ precisely at the point $(0,\ldots,0,\alpha_{i_0},0,\ldots,0)$ by Lemma \ref{lemma-nm}). If the matrix has two columns $C_{i_0}$ and $C_{i_0'}$ of the above form, with $i_0,i_0'\in\{1,\ldots,d\}$ and $i_0\not=i_0'$, then $\Vol_{|I|}({S}'_{I,r})$ is not involved in the linear part \eqref{lpnn} of $\nu(f_d')$.
Now, if it has one column $C_{i_0}$ for some $i_0\in\{1,\ldots,d\}$ and no any other column $C_{i_0'}$ for $i_0'\in\{1,\ldots,d\}\setminus \{i_0\}$, then $\Vol_{|I|}({S}'_{I,r})$ contributes to the term $\alpha_{i_0}P_{i_0}$ which appears in \eqref{lpnn}. Note that in the latter case, the matrix cannot have any column of the form $C_{i_j}$ with $i_j\in\{i_1,\ldots, i_p\}$ (as otherwise the constant polynomial $P_{i_0}$ would depend on $\alpha_{i_j}$). Altogether, for any $1\leq i_0\leq d$, the volume $\Vol_{|I|}({S}'_{I,r})$ contributes to the term $\alpha_{i_0}P_{i_0}$ if and only if the corresponding matrix has a column of the form $C_{i_0}$ and no column of the form $C_{i_0'}$ for any other $i_0'\in\{1,\ldots,d,i_1,\ldots, i_p\}\setminus \{i_0\}=J\setminus \{i_0\}$. In other words, $\Vol_{|I|}({S}'_{I,r})$ contributes to the term $\alpha_{i_0}P_{i_0}$ if and only if ${S}'_{I,r}\in \Xi'_{I,J,i_0}$. Thus,
\begin{align*}
\alpha_{i_0}P_{i_0} & = \sum _{I \subseteq \{1,\ldots,n\},\, I\ni i_0} 
\Bigg(\sum_{{S}'_{I,r}\in {\Xi}'_{I,J,i_0}} 
(-1)^{n-|I|} |I|! \Vol _{|I|} ({S}'_{I,r})\Bigg)\\
 & = \sum _{I \subseteq \{1,\ldots,n\},\, I\ni i_0} \alpha_{i_0}
\Bigg(\sum_{\widetilde{S}'_{I,r}\in \widetilde{\Xi}'_{I,J,i_0}} (-1)^{n-|I|} |I|! \Vol _{|I|} (\widetilde{S}'_{I,r})\Bigg)\\
 & = \alpha_{i_0} \widetilde{\nu}_{\Xi',J,i_0}(f'_d),
\end{align*}
where $\widetilde{S}'_{I,r}$ denotes the reduced simplex associated to $S'_{I,r}$ (see Section \ref{sect-rnn}).
Since $\Gamma(f_d')$ is obtained from $\Gamma(f_d)$ only by ``adding'' the vertices \begin{align*}
v_{i_j}:=(0,\ldots,0,\alpha_{i_j},0,\ldots,0)
\end{align*}
(with $\alpha_{i_j}$ at the $i_j$th place) for large $\alpha_{i_j}$ ($1\leq j\leq p$), if a simplex ${S}'_{I,r}$ of ${\Xi}'_I$ with maximal dimension is not a simplex of ${\Xi}_I$ (in particular this is the case if $\Xi_I=\emptyset$), then necessarily it intersects the $i_j$th coordinate axis of $\mathbb{R}^n _+$ for some $j$ ($1\leq j\leq p$). It follows that 
\begin{align*}
\Xi_{I,\{1,\ldots,d\},i_0}=\Xi'_{I,J,i_0},
\end{align*}
and hence, 
\begin{align*}
\widetilde{\nu}_{\Xi,i_0}(f_d):=\widetilde{\nu}_{\Xi,\{1,\ldots,d\},i_0}(f_d)=\widetilde{\nu}_{\Xi',J,i_0}(f'_d)=P_{i_0}.
\end{align*}
Since the choice of $\Xi$ is arbitrary and $P_{i_0}$ is a constant independent of $\Xi$, it follows that the modified Newton number $\widetilde{\nu}_{\Xi,i_0}(f_d)$ is also independent of $\Xi$. The notation $\widetilde{\nu}_{i_0}(f_d):=\widetilde{\nu}_{\Xi,i_0}(f_d)$ is therefore quite relevant.

Since the volume $\Vol_{|I|}({S}'_{I,r})$ contributes to the term $P_0$ if and only if the simplex ${S}'_{I,r}$ belongs to $\Xi'_{I,J,0}$ (we recall that $P_0$ is constant, independent of $\alpha_{i_1},\ldots,\alpha_{i_p}$), a similar argument shows that 
\begin{displaymath}
\nu_{0}(f_d)+(-1)^n=\nu_{\Xi',J,0}(f'_d)+(-1)^n=P_{0}. \qedhere
\end{displaymath}
\end{proof}

Now we prove Lemma \ref{lemma-nm}.

\begin{proof}[Proof of Lemma \ref{lemma-nm}]
We argue by contradiction. Suppose that $f$ has a term of the form $c_i z_i^{a_i}$ for some $i$ ($1\leq i\leq d$). To simplify, without loss of generality, we may assume $i=1$ (the other cases are similar). By the Iomdine-L\^e-Massey formula again and by Lemmas 3.7 and 3.8 of \cite{BO}, for any $0\ll b_1\ll \cdots\ll b_d$ sufficiently large (in particular so that $a_1<b_1$), the function
\begin{equation*}
g(z) := f(z)+z_1^{b_1}+ \cdots + z_d^{b_d}
\end{equation*}
is non-degenerate and has an isolated singularity at $0$. Then, by \cite[Corol\-lary~2.9]{BO}, its support (denoted by $\supp g$) satisfies so-called \emph{Kouchnirenko condition} (see \cite{K2} or Section~2 of \cite{BO} for the definition; see also Section 3 of \cite{O} and the references mentioned therein for equivalent formulations and historical comments).
Now, since $a_1<b_1$, the Newton diagrams of $g$ and of the function
\begin{equation*}
g'(z):=g(z)-z_1^{b_1}
\end{equation*}
coincide. It follows that $g'$ is also non-degenerate and such that its support $\supp g'$  satisfies the Kouchnirenko condition. Theorem 3.1 of \cite{BO} then implies that $g'$ has an isolated singularity at $0$. 
If $d:=\dim_0\Sigma f=1$, then this is already a contradiction, because in this case $g'=f$. If $d>1$, then define
\begin{align*}
Z_{g'} & := \left \{ z \in \cc^n \, ;\, \frac{\partial g'}{\partial z_i}(z)=0 \mbox{ for all } i\in\{1,d+1,\ldots,n\}\right \}.
\end{align*}
Define $Z_f$ similarly (replacing $g'$ by $f$).
Clearly, $Z_{g'}=Z_f$. Therefore, we have $d:=\dim_0 \Sigma f \leq \dim_0 Z_{f}=\dim_0 Z_{g'}$, but since $g'$ has an isolated singularity at~$0$, we must also have $\dim_0 Z_{g'}=d-1$, a new contradiction.
\end{proof}

\section{Proof of Corollary \ref{cor2}}\label{pcor2}

First, we show that $\dim_0 \Sigma f =\dim_0 \Sigma g$. The argument is similar to that given in the proof of Lemma \ref{lemma-nm}.
We argue by contradiction. Put $d:=\dim_0 \Sigma f $ and $s:=\dim_0 \Sigma g$, and suppose for instance $d<s$. By the Iomdine-L\^e-Massey formula  and by Lemmas 3.7 and 3.8 of \cite{BO}, for any integers $0\ll \alpha_1\ll\cdots\ll\alpha_d$ sufficiently large\footnote{Precisely, $\alpha_p>\max\{2,\rho_{f_{p-1},z^{(p-1)}}(0),\rho_{g_{p-1},z^{(p-1)}}(0),m(f_{p-1}),m(g_{p-1})\}$.}, the functions 
\begin{equation*}
f_d(z):=f(z)+z_1^{\alpha_1}+\cdots+z_d^{\alpha_d}
\quad\mbox{and}\quad
g_d(z):=g(z)+z_1^{\alpha_1}+\cdots+z_d^{\alpha_d}
\end{equation*}
are non-degenerate, $f_d$ has an isolated singularity at $0$, and $\dim_0\Sigma g_d=s-d>0$. Then, by \cite[Corollary 2.9]{BO}, the support of $f_d$ satisfies the Kouchnirenko condition. Since $f$ and $g$ have the same Newton diagram, it follows that $\Gamma(f_d)=\Gamma(g_d)$ too. Thus the support of $g_d$ also satisfies the Kouchnirenko condition, and by [1, Theorem 3.1], the function $g_d$ must have an isolated singularity at $0$ --- a contradiction.

Now, to show that the L\^e numbers of $f$ and $g$ at $0$ with respect to the coordinates $z=(z_1,\ldots,z_n)$ are equal, we apply Theorem \ref{mt}. By this theorem, these L\^e numbers are described in terms of the modified (and special modified) Newton numbers of the functions $f_d$ and $g_d$. Then the result follows immediately from the equality $\Gamma(f_d)=\Gamma(g_d)$.

\appendix
\section{}\label{ppmt}

For completeness, in this appendix, we give a proof of a useful elementary lemma which we have used in the proof of Theorem \ref{mt}.

Let $d$ be a positive integer. Consider the following system $\mathscr{S}$ of $d$ integral inequalities with $d$ variables $\alpha_1,\ldots,\alpha_d$:
\begin{equation*}\label{system}
\begin{cases}
\alpha_1 \geq c_1,\\
\alpha_2 \geq c_2(\alpha_1),\\
\alpha_3 \geq c_3(\alpha_1,\alpha_2),\\
\cdots \\
\alpha_d \geq c_d(\alpha_1,\ldots,\alpha_{d-1}).
\end{cases}
\end{equation*}
Here, $c_1$ is a constant, and for $2\leq \ell\leq d$, $c_\ell(\alpha_1,\ldots,\alpha_{\ell-1})$ is a number depending on $\alpha_1,\ldots,\alpha_{\ell-1}$. For each $1\leq r\leq d$, let $\mathscr{S}(r)$ be the system consisting only of the first $r$ inequalities of the system $\mathscr{S}$. Finally, let $Z(r) \subseteq \zz^{r}$ be the set of (integral) solutions of the system $\mathscr{S}(r)$.

\begin{lem}\label{lemma-ul}
For any $1\leq r\leq d$, if $P(x_1,\ldots,x_r)$ is a polynomial function that vanishes on $Z(r)$, then it is identically~zero. 
\end{lem}

\begin{proof}
By induction on $r$. For $r=1$, the lemma immediately follows from the fundamental theorem of algebra. Now suppose the lemma holds true for some integer $r-1$ (with $r\geq 2$), and let us show that it also holds true for the integer $r$. So, let $P(x_1,\ldots,x_r)$ be a polynomial function such that $P(\alpha_1,\ldots,\alpha_r)=0$ for any $(\alpha_1,\ldots,\alpha_r) \in Z(r)$. Note that $(\alpha_1,\ldots,\alpha_r) \in Z(r)$ implies $(\alpha_1,\ldots,\alpha_{r-1}) \in Z(r-1)$. Expand $P$ with respect to the variable $x_r$:
\begin{equation*}
P(x_1,\ldots,x_r)=\sum _{k=0}^{\delta} P_k(x_1,\ldots,x_{r-1})\, x_r^k.
\end{equation*}
(Here, $\delta$ denotes the degree of $P$.)
Then for all $(\alpha_1,\ldots,\alpha_r) \in Z(r)$,
\begin{equation*}
\sum _{k=0}^\delta P_k(\alpha_1,\ldots,\alpha_{r-1})\, \alpha_r^k=0.
\end{equation*}
By the fundamental theorem of algebra, 
it follows that for each $0\leq k\leq \delta$, 
\begin{equation*}
P_k(\alpha_1,\ldots,\alpha_{r-1})=0 
\end{equation*}
for every fixed $(\alpha_1,\ldots,\alpha_{r-1})\in Z(r-1)$. Now, by the induction hypothesis, this implies that the polynomial $P_k$ identically vanishes.
\end{proof}

\end{document}